\documentclass[12pt,A4]{article}

\DeclareSymbolFont{AMSb}{U}{msb}{m}{n}
\DeclareSymbolFontAlphabet{\Bbb}{AMSb}

\usepackage[ansinew]{inputenc} 
\usepackage{latexsym}

\usepackage{amstext,amsfonts,amsmath,amsthm,graphicx,amssymb,amscd,epsfig}
\usepackage{rotating} 
\usepackage{xcolor} 
 \usepackage{soul} 

\newtheorem{theorem}{Theorem}[section]

\newtheorem{lemma}[theorem]{Lemma}
\newtheorem{corollary}[theorem]{Corollary}
\newtheorem{proposition}[theorem]{Proposition}

\newtheorem{example}[theorem]{Example}

\newcommand{\norm}[1]{\left\Vert#1\right\Vert}
\newcommand{\abs}[1]{\left\vert#1\right\vert}

\newcommand{\C}{\mathbb{C}}
\newcommand{\R}{\mathbb{R}}
\newcommand{\Z}{\mathbb{Z}}
\newcommand{\N}{\mathbb{N}}

\newcommand{\spec}{\operatorname{{Spec}}}

\newcommand{\U}{{\mathcal U}}

\newcommand{\Emb}{\textrm{Emb}}

\newcommand{\Hom}{\textrm{Hom}}
\newcommand{\Fix}{\textrm{Fix}}
\newcommand{\Per}{\textrm{Per}}

\begin{document}

\begin{center}
{\large {\bf Global Dynamics for Symmetric Planar Maps}}\\
\mbox{} \\
\begin{tabular}{ccc}
{\bf Bego\~na Alarc\'on$^{\mbox{a,c}}$} & {\bf Sofia B.\ S.\ D.\ Castro$^{\mbox{a,b}}$} & {\bf Isabel S.\ Labouriau$^{\mbox{a}}$}
\end{tabular}
\end{center}

\bigbreak
\noindent {\small $^{\mbox{a}}$} Centro de Matem\'atica da Universidade do Porto, Rua do Campo Alegre 687, 4169-007 Porto, Portugal.

\noindent {\small $^{\mbox{b}}$} Faculdade de Economia do Porto, Rua Dr. Roberto Frias, 4200-464 Porto, Portugal.

\noindent {\small $^{\mbox{c}}$} permanent address:
Departament of Mathematics. University of Oviedo; Calvo Sotelo s/n; 33007 Oviedo; Spain.

\bigbreak

\begin{abstract}

We consider sufficient conditions to determine the global dynamics for equivariant maps of the plane with a unique fixed point which is also hyperbolic. When the map is equivariant under the action of a compact Lie group, it is possible to describe the local dynamics.
 In particular, if the group contains a reflection, there is a line invariant by the map. This allows us to use results  based on the theory of free homeomorphisms to describe the global dynamical behaviour. 
We briefly discuss the case when reflections are absent, for which global dynamics may not follow from local dynamics near the unique fixed point.

\end{abstract}
\medskip


\newpage

\section{Introduction}

Dynamics of planar maps has drawn the attention of many authors.
See, for instance, references such as \cite{Alarcon}, \cite{Brown}, \cite{Cima-Manosa}, \cite{franks} or \cite{LeRoux}.
Problems addressed include the theory of free homeomorphism and trivial dynamics, searching for sufficient conditions
for global results with topological tools or the Discrete Markus-Yamabe Problem.
The latter conjectures results on global stability from local stability of the fixed point under some additional condition on the Jacobian matrix\footnote{The Discrete Markus-Yamabe Problem also derives its fame from its relation to the Jacobian Conjecture \cite{vanDenEssen}.}.
Results in \cite{Alarcon} guarantee the global stability of the unique fixed point  resorting to local conditions and the existence of an invariant embedded curve joining the fixed point to infinity.

To the best of our knowledge, this problem has been addressed exclusively in a non-symmetric context.
However the existence of the invariant curve in \cite{Alarcon}
made it seem natural to approach this problem in a symmetric setting,
 where invariant spaces for the dynamics are a key feature.
In this context, we address the global dynamics of planar diffeomorphisms having a unique fixed point which is hyperbolic.
We restrict our attention to the cases where the fixed point is either  an attractor or a repellor.
The case of a saddle point does not rely so much on symmetry.
It will therefore be addressed elsewhere.

The issue of uniqueness of a fixed point has been addressed by Alarc\'on {\em et al}
\cite{Alarcon-orbitas-periodicas} who gave simple conditions for planar maps under which the origin is the unique fixed point.

The presence of symmetry constrains  the admissible local dynamics near the fixed point.
We extend, whenever possible, the local dynamics to the whole plane using the properties determined by the symmetry.
The reader may see that, in the case of $O(2)$, $\Z_2(\langle\kappa\rangle)$ and $D_n$,
 local dynamics determines global dynamics.
However, $SO(2)$ and $\Z_n$ symmetry do not provide any extra information on how to go from local to global dynamics. Actually, for $\Z_n$, we have constructed examples where more than one configuration can occur.
For instance, in the case of a local attractor, the global dynamics may exhibit either several periodic points \cite{SofisbeSzlenk} or a globally attracting set with special properties
 \cite{AlarconDenjoy}. 
For $SO(2)$ the dynamics is determined by a one-dimensional map.
These situations provide a comprehensive description of the possible global dynamics.

The paper is organised as follows: in the next section we transcribe preliminary results concerning dynamics of planar maps and equivariance. These are organised in two subsections, the first on Topological Dynamics and the second on Equivariance. The reader familiar with the subject may skip them without loss. 
Section~\ref{secLocalGlobal} shows how  local results  can be extended to global results.

We thus establish all possible dynamics in the presence of
symmetry. This is discussed at the end in Section~\ref{Discussion}. Note that the aforementioned groups are the only compact
subgroups of $O(2)$ acting on the plane. For the reader's
convenience the results obtained are summarised in the Equivariant
Table of Appendix~\ref{EquivTable}.

\section{Preliminaries} \label{secPre}

This section contains definitions and known results about topological dynamics and equivariant theory. These are grouped in three separate subsections, which are elementary for readers in each field.

\subsection{Topological Dynamics}

We consider planar topological embeddings, that is, continuous and injective maps defined in $\R^ 2$. The set of topological embeddings of the plane is denoted by $\Emb(\R^2)$.

Recall that for $f\in \Emb(\R^2)$ the equality $f(\R^2)=\R^2$ may not hold.
Since every map $f\in \Emb(\R^2)$ is open (see \cite{libroembeddings}), we will say that $f$ is a homeomorphism if $f$ is a topological embedding defined onto $\R^2$.
The set of homeomorphisms of the plane will be denoted by $\Hom(\R^2)$.

When $\mathcal{H}$ is  one of these sets  we denote by $\mathcal{H}^ {+}$ (and $\mathcal{H}^ {-}$) the subset of orientation preserving (reversing) elements of $\mathcal{H}$.

Given a continuous map $f: \R^2 \to \R^2$, we say that $p$ is a \emph{non-wandering point} of $f$ if for every neighbourhood $U$ of $p$ there exists an integer $n>0$ and a point $q\in U$ such that $f^n(q)\in U$. We denote the set of non-wandering points by $\Omega(f)$. We have
$$
\Fix(f)\subset \Per(f) \subset \Omega(f),
$$
where $\Fix(f)$ is the set of fixed points of $f$,  and $\Per(f)$ is the set of periodic points of $f$.

Let $\omega(p)$ be the set of points $q$ for which there is a
sequence $n_j\to+\infty$ such that $f^{n_j}(q)\to p$. If $f\in
\Hom(\R^2)$ then $\alpha(p)$ denotes the set $\omega(p)$ under
$f^{-1}$.

Let $f\in \Emb(\R^2)$ and $p\in \R^2$.  We say that $\omega(p)=\infty$ if $\norm{f^n(p)}\to \infty$ as $n$ goes to $ \infty$. Analogously, if $f\in \Hom(\R^2)$, we say that $\alpha(p)=\infty$ if $\norm{f^{-n}(p)}\to \infty$ as $n$ goes to $ \infty$.

We say that $0\in \Fix(f)$ is a \emph{local attractor} if its basin of attraction $\U=\{p\in \R^ 2 : \omega(p)=\{0\}\}$ contains an open neighbourhood of $0$ in $\R^2$ and that $0$ is a \emph{global attractor} if $\U=\R^2$. The origin is a \emph{stable fixed point} if for every neighborhood $U$ of $0$ there exists another neighborhood $V$ of $0$ such that $f(V)\subset V$ and $f(V)\subset U$. Therefore, the origin is an \emph{asymptotically local (global) attractor} or a \emph{(globally) asymptotically stable fixed point} if it is a stable local (global) attractor. See \cite{Bhatia} for examples.

We say that $0\in \Fix(f)$ is a \emph{local repellor} if there exists a neighbourhood $V$ of $0$ such that $\omega(p)\notin V$ for all $0\neq p\in \R^ 2$ and a \emph{global repellor} if this holds for $V=\R^2$.
 We say that the origin is an \emph{asymptotically global repellor} if it is a  global repellor
 and, moreover, if for any neighbourhood $U$ of $0$ there exists another neighbourhood $V$ of $0$, such that,
 $V\subset f(V)$ and $V\subset f(U)$.

When the origin is a fixed point of a $C^ 1-$map of the plane we say the origin is a \emph{local saddle} if the two eigenvalues of $Df_0$, $\alpha, \beta$, are both real and verify $0<\abs{\alpha}<1<\abs{\beta}$.

We also need the following theorem of Murthy \cite{Murthy}, to be applied to parts of the domain of our maps with no fixed points:

\begin{theorem}[Murthy \cite{Murthy}] \label{teoMurthy} Let $f\in \Emb^ {+}(\R^ 2)$. If $\Fix(f)=\emptyset$, then $\Omega(f)=\emptyset$.
\end{theorem}

We say that $f\in \Emb(\R^2)$ has \emph{trivial dynamics} if, for all $p\in \R^ 2$, either  $\omega(p) \subset \Fix(f)$ or $\omega(p)=\infty$. Moreover, we say that a planar homeomorphism has trivial dynamics if, for all $p\in \R^ 2$, both $\omega(p), \alpha(p) \subset \Fix(f)\cup\{\infty\}$.

Let $f: \mathbb{R}^N \rightarrow \mathbb{R}^N$ be a continuous map.
Let $\gamma : [0,\infty) \to \mathbb{R}^2$  be a topological
embedding of $[0,\infty) \,. \;$  As usual, we identify $\gamma$
with $\gamma\,([0,\infty))\,.$ We will say that $\gamma$ is an \emph{
$f-$invariant ray} if $\, \gamma(0)=(0,0)\,, \,$ $\,f(\gamma)\subset
\gamma \,, \,$ and $\lim_{t\to\infty}|\gamma(t)|=\infty$, where
$|\cdot| \,$ denotes the usual Euclidean norm.

\begin{proposition}[Alarc\'on {\em et al.} \cite{Alarcon}] \label{lemrayo} Let $f\in \Emb^ {+}(\R^2)$ be such that $\Fix(f)=\{0\}$. If there exists an $f-$invariant ray $\gamma$, then $f$ has trivial dynamics.
\end{proposition}

The relation between the stability of the origin and the admissible forms of the Jacobian at that point is well known, but it is not usually clear that it holds for continuous maps that are
not necessarily $C^1$. The precise hypotheses are stated in the
following result, and its proof is given in Appendix~\ref{ApendiceEstabilidade}.

\begin{proposition}\label{Liapunov}
Let $U \subset \R^N$ be an open set containing the origin and let $f: \; \R^N \rightarrow \R^N$ be a continuous map, differentiable at the origin and such that $f(0)=0$. If all the eigenvalues of $Df(0)$ have norm strictly smaller than one then the origin is locally assymptotically stable. If all the eigenvalues of $Df(0)$ have norm strictly greater than one then the origin is a local repellor.
\end{proposition}

\subsection{Equivariant Planar Maps}

Let $\Gamma$ be a compact Lie group acting on $\R^2$. The
following definitions and results are taken from Golubitsky {\em
et al.} \cite{golu2},  especially Chapter~XII,  to which we refer
the reader interested in further detail.

Given a map $f:\R^2\longrightarrow\R^2$,
we say that $\gamma \in \Gamma$ is a \emph{symmetry} of $f$ if $f(\gamma x)=\gamma f(x)$.
We define the \emph{symmetry group} of $f$ as the smallest closed subgroup of $GL(2)$ containing all the symmetries of $f$. It will be denoted by $\Gamma_f$.

We say that $f:\R^2\to \R^2$ is  \emph{$\Gamma-$equivariant} or that $f$ {\em commutes} with $\Gamma$ if
$$
f(\gamma x)=\gamma f(x) \quad \text{ for all }\quad \gamma \in \Gamma.
$$

It follows that every map $f:\R^2\to \R^2$ is equivariant under the action of its symmetry group, that is, $f$ is $\Gamma_f-$equivariant. 
We are interested in maps having a non-trivial symmetry group, $\Gamma_f\subset GL(2)$.

Let  $\Sigma$ be a subgroup of $\Gamma$. The {\em fixed-point subspace} of $\Sigma$ is
$$
\Fix (\Sigma) =\{p\in \R^2: \sigma p=p \; \text{ for all } \; \sigma \in \Sigma\}.
$$
If $\Sigma$ is generated by a single element $\sigma \in \Gamma$, we write \emph{$\Fix\langle\sigma\rangle$} instead of $\Fix (\Sigma)$.

We note that, for each subgroup $\Sigma$ of $\Gamma$, $\Fix (\Sigma)$ is invariant by the dynamics of a $\Gamma-$equivariant map (\cite{golu2}, XIII, Lemma 2.1).

For a group $\Gamma$ acting on $\R^2$ a non trivial fixed point subspace arises when $\Gamma$ contains a reflection. By a linear change of coordinates we may take the reflection to be the {\em flip}
    $$
    \kappa . (x,y) = (x, -y) .
    $$

\section{Symmetric Global Dynamics} \label{secLocalGlobal}

In this section we study the global dynamics of a symmetric discrete dynamical system with a unique and hyperbolic attracting fixed point.
We establish conditions for the hyperbolic local dynamics 
to become global dynamics.
For most symmetry groups, the local dynamics is restricted to either an attractor or a repellor.
Saddle points only occur for very small symmetry groups and therefore the study of these points does not depend so much on symmetry issues and requires additional  tools.
As pointed out before, this will be the object of a separate article.

We  address the dynamics when the groups involved possess an element acting as a reflection (flip). This is the main result of this article. In this case, we make use of the fact that there exists an invariant ray for the dynamics (either of $f$ itself or of $f^2$) from which results follow.

We begin with two convenient results.
Although the next lemma is only required to hold for planar maps, we present it for $\R^N$, as the proof is the same.

Let $p\in \R^N$ and $f:\R^N \to \R^N$ be a continuous map.
We denote by $\omega_2(p)$  the $\omega-$limit of $p$ with respect to $f^2$, given by
 $$
 \omega_2(p)=\{q\in \R^N : \lim \;f^{2n_k}(p)=q, \; \text{ for some sequence} \; n_k\to \infty \} .
 $$

\begin{lemma} \label{lemOmega2} Let $f\in \Emb(\R^N)$  be such that $f(0)=0$.

For $p\in \R^N$,

\begin{enumerate}
\item[a)]  if $\omega_2(p)=\{0\}$, then   $\omega(p)=\{0\}$;
\item[ b)] if $\omega_2(p)=\infty$,   then
$\omega(p)=\infty$.
\end{enumerate}

\end{lemma}

\begin{proof} Let $p\in \R^N$.

$a)$ Suppose $\omega_2(p)=\{0\}$ and suppose also that $\omega(p)\neq\{0\}$.
Then there exists an $r\neq 0$ such that
$r\in\omega(p)$. In that case $r=\lim f^{n_k}(p)$, so there exists
a $k_0\in \N$ such that $\forall k>k_0$, $n_k$ is odd because
$\omega_2(p)=\{0\}$. Then, $f(r)=\lim f^{n_k+1}(p)$ with $n_k +1$
even. So $f(r)\in \omega_{2}(p) $ hence $f(r)=0$ with $r\neq 0$,
which is impossible because $f$ is an injective map such that $f(0)=0$.

$b)$ Suppose now $\omega_2(p)= \infty$ and also that
$\omega(p)\neq\infty$. Then there exists an $r\in
\R^N$ such that $r\in\omega(p)$. In that case $r=\lim f^{n_k}(p)$,
so there exists a $k_0\in \N$ such that $\forall k>k_0$, $n_k$ is
odd because  $\omega_2(p)=\infty$. Then, $f(r)=\lim
f^{n_k+1}(p)$ with $n_k +1$ even. So $f(r)\in \omega_{2}(p)$ which is impossible, since $\omega_2(p)=\infty$.
\end{proof}

\begin{lemma} \label{lemgdim1} Let $g:[0,1) \to [0,1)$ be a continuous and injective map such that $Fix(g)=\{0\}$. The following holds:
\begin{itemize}
\item[$a)$] If $0$ is a local attractor for $g$, then $0$ is a global attractor for $g$.
\item[$b)$] If $0$ is a local repellor for $g$, then $0$ is a global repellor for $g$. \end{itemize}
\end{lemma}

\begin{proof} Assume $0$ is a local attractor. Since $g$ is a continuous map, $g$ is increasing at $0$. Because $g$ is injective, $g$ is increasing in $[0,1)$. Moreover, $\Fix(g)=\{0\}$ so the graph of $g$ does not cross the diagonal of the first quadrant and one of the following holds:

\begin{itemize} \item[i)] $g(x)>x$, for all $x\in (0,1)$;
\item[ii)] $g(x)<x$, for all $x\in (0,1)$.
\end{itemize}

Only $ii)$ can happen when $0$ is a local attractor. Then $g(x)<x$, for all $x\in [0,1)$ and $0$ is a global attractor for $g$.

The proof of b) follows in a similar fashion.
\end{proof}

\medskip

We now proceed to use the existence of an invariant ray to obtain information concerning the dynamics.

\begin{lemma} \label{lemFixdim1} Let $f:\R^2\to \R^2$ be a map with symmetry group $\Gamma$. If $\kappa \in \Gamma$, then $\Fix \langle\kappa\rangle$ is an  $f-$invariant line. Moreover, $\Fix \langle\kappa\rangle$ contains an $f^2-$invariant ray.
\end{lemma}

\begin{proof} By Lemma XIII, $2.1$ and Theorem XIII, $2.3$ in \cite{golu2}, $\Fix \langle\kappa\rangle$ is a vector subspace of dimension one such that $f(\Fix \langle\kappa\rangle) \subseteq \Fix \langle\kappa\rangle$. Let $\gamma$ denote one of the two half-lines in $\Fix \langle\kappa\rangle$, then $\gamma$ is an $f^2-$invariant ray.
\end{proof}

The next proposition describes the admissible $\omega$-limit set of a point and is essential for the main results.

\begin{proposition} \label{proprayoS} Let $f\in \Emb(\R^2)$ have symmetry group $\Gamma$ with $\kappa \in \Gamma$, such that $\Fix(f)=\{0\}$. Suppose one of the followings holds:

\begin{itemize}
\item[$a)$] $f\in \Emb^ {+}(\R^2)$ and $f$ does not interchange connected components of $\R^2 \setminus \Fix \langle\kappa\rangle$.
\item[$b)$]  $\Fix(f^ 2)=\{0\}$.
\end{itemize}
Then for each $p\in \R^2$ either $\omega(p)=\{0\}$ or $\omega(p)=\infty$.
\end{proposition}

\begin{proof} Suppose $a)$ holds. By Lemma \ref{lemFixdim1}, $\R^2\setminus \Fix \langle\kappa\rangle$ is the disjoint union of two open subsets $U_1, U_2\subset \R^2$ homeomorphic to $\R^2$. Moreover,
 $f|_{U_{i}}: U_{i} \to U_{i}$  for  $i=1,2$ is an orientation preserving embedding without fixed points. Then by Theorem \ref{teoMurthy}, $\Omega(f|_{U_{i}})=\emptyset$ for $i=1,2$ and it then follows that $\; \Omega(f) \subseteq \Fix \langle\kappa\rangle$. \par

The subspace $\Fix \langle\kappa\rangle\setminus\{0\}$  is the disjoint union of two  subsets  homeomorphic to $(0,1)$.
Even if  $f$ interchanges these components,  $f^2$ does not.
Then applying  Lemma \ref{lemgdim1}  to the restriction of $f^2$ to each component, it follows that for $p\in \Fix \langle\kappa\rangle$ either $\omega_2(p)=0$ or $\omega_2(p)=\infty$.
By Lemma \ref{lemOmega2} , it follows that for $p\in \Fix \langle\kappa\rangle$ either $\omega(p)=0$ or $\omega(p)=\infty$.

Let $p\in \R^2 \setminus \Fix \langle\kappa\rangle $. Since $\; \Omega(f) \subseteq \Fix \langle\kappa\rangle$, we have that $\omega(p)\subseteq \Fix \langle\kappa\rangle$. We show next that $\omega(p)\neq \Fix \langle\kappa\rangle$.

Suppose there is an $r\in\omega(p)\cap ({\Fix \langle\kappa\rangle}\setminus\{0\})$ and an open
neighbourhood $K$ of $r$  such that $0\notin K$ and $\Fix \langle\kappa\rangle \cap K$ is an embedded
segment and $K\setminus\Fix \langle\kappa\rangle$ is the union of two disjoint disks
$W_1$ and $W_2$ homeomorphic to $\R^2$. Suppose without loss of generality that $p\in U_1$, then the positive orbit of $p$ accumulates at $r$ and this positive orbit meets $W_1$ infinitely many times. Since $r\in \Omega(f)\setminus \{0, \infty\}$ is not a fixed point, taking $K$ (and hence $W_1$) sufficiently small, there exists an open disk $V\subset W_1$  and a positive integer $n,$ with $n\ge 2$,
such that for some $s\in V,$ we have that $f^n(s)\in V$, while
$V\cap f^{\ell}(V)=\emptyset $, for $\ell=1,2,\dots, n-1$. Then, by Theorem 3.3 in \cite{Murthy}, $f$ has a fixed point in $V$ which contradicts the uniqueness of the fixed point. So the orbit of $p$ does not accumulate at $\Fix \langle\kappa\rangle$ and hence either $\omega(p)=0$ or $\omega(p)=\infty$.

Suppose $b)$ holds. By Lemma \ref{lemFixdim1} there exists an $f^2-$invariant ray $ \gamma \subset \Fix \langle\kappa\rangle$. Moreover, $f^2\in \Emb^{+}(\R^2)$ and $\Fix(f^2)=\{0\}$, so by Proposition~\ref{lemrayo} we have that for each $p\in \R^2$ either $\omega_2(p)=\{0\}$ or $\omega_2(p)=\infty$ and therefore, by Lemma \ref{lemOmega2}, either $\omega(p)=\{0\}$ or $\omega(p)=\infty$.
\end{proof}

The next example shows that assumption $b)$ in Proposition \ref{proprayoS} is necessary in the case where $f$ interchanges connected components of $\R^2 \setminus \Fix \langle\kappa\rangle$.

\begin{example} Consider the map $f: \R^2 \to \R^2$ defined by
$$
f(x,y)=\left(-ax^3+(a-1)x,-\frac{y}{2}\right) \quad 0<a<1.
$$
It is easily checked that $f$ has symmetry group $D_2$ and verifies (see Figure \ref{figureattractor}):
\begin{enumerate}
\item $f\in \Emb^{+}(\R^2)$ is an orientation-preserving diffeomorphism.
\item $\spec(f)\cap [0, \infty)=\emptyset$.
\item $0$ is a local hyperbolic attractor.
\item $\Fix(f^2)\neq \{0\}$.
\end{enumerate}
\begin{figure}[hh]
\centering
\includegraphics[scale=0.5]{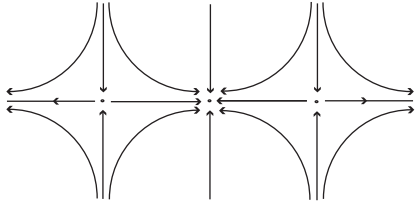}
\caption{A local attractor which is not a global attractor due to the existence of periodic orbits.} \label{figureattractor}
\end{figure}
\end{example}

The main results, Theorems \ref{proprayoS1} and \ref{proprayoS2}, are obtained as a direct consequence of Proposition \ref{proprayoS} under additional assumptions.

We say that a map $f$ is \emph{dissipative} if there is an open set $V$ such that $\R^2\backslash V$ is compact
and $\omega(p)\notin V$ for all $p\in\R^2$.

\begin{theorem} \label{proprayoS1} Let $f\in \Emb(\R^2)$ be dissipative with symmetry group $\Gamma$ with $\kappa \in \Gamma$ such that $\Fix(f)=\{0\}$. Suppose in addition that one of the following holds:

\begin{itemize}
\item[$a)$] $f\in \Emb^ {+}(\R^2)$ and $f$ does not interchange connected components of $\R^2 \setminus \Fix \langle\kappa\rangle$.
\item[$b)$]  There exist no $2-$periodic orbits.
\end{itemize}
Then $0$ is a global attractor.

\end{theorem}

\begin{proof} Follows from 
Proposition \ref{proprayoS} since being dissipative excludes $\omega(p)=\infty$.
\end{proof}

\begin{corollary} Suppose the assumptions of Theorem \ref{proprayoS1}
are verified and $f$ is differentiable at $0$. If every eigenvalue of $Df(0)$ has norm strictly less than one, then $0$ is a global asymptotic attractor.
\end{corollary}

\begin{proof} Follows by Theorem \ref{proprayoS1} and Proposition \ref{Liapunov}.
\end{proof}

\begin{theorem} \label{proprayoS2} Let $f\in \Emb(\R^2)$ be a
map with symmetry group $\Gamma$ with $\kappa \in \Gamma$ such that $\Fix(f)=\{0\}$. Suppose in addition that one of the following holds:

\begin{itemize}
\item[$a)$] $f\in \Emb^ {+}(\R^2)$ and $f$ does not interchange connected components of $\R^2 \setminus Fix\langle\kappa\rangle$.
\item[$b)$] There exist no $2-$periodic orbits.
\end{itemize}
Then, if $0$ is a local repellor, then $0$ is a global repellor.

\end{theorem}

\begin{proof} Follows from 
Proposition \ref{proprayoS}, since a local repellor excludes $\omega(p)=\{0\}$.
\end{proof}

\begin{corollary} Suppose the assumptions of Theorem \ref{proprayoS2}
are verified and $f$ is differentiable at $0$. If every eigenvalue of $Df(0)$ has norm strictly greater than one, then $0$ is a global asymptotic repellor.
\end{corollary}

\begin{proof} Follows from Theorem \ref{proprayoS2} and Proposition \ref{Liapunov}.
\end{proof}

\section{Discussion: From Local to Global Symmetric Dynamics}\label{Discussion}

In this section we discuss  all possible groups of symmetries for a planar topological embedding,
and how this may be used to obtain global dynamics from local information near a unique fixed point. 
In our discussion of symmetries we need  only be concerned with compact subgroups of the orthogonal group $O(2)$,
since every compact Lie group in $GL(2)$ can be identified with one of them. For the convenience of the reader less familiar with symmetry, these are listed in Appendix \ref{subgroups}, together with their action on $\R^2$.
Finally we give a brief description of what happens for $SO(2)$ and  $\Z_n$, the groups that do not contain a flip, where the information about the local dynamics cannot be extended to global dynamics.

Since most of our results depend on the existence of a unique
fixed point for $f$, we assume also that the group action is {\em faithful}: for every 
$\gamma\in\Gamma$, different from the identity, there exists $x\in\R^2$ such that $\gamma x\ne x$.
Therefore $\Fix (\Gamma) = \{ 0
\}$ and hence, if $f$ is $\Gamma-$equivariant then $f(0)=0$. 

The Jacobian matrix of an equivariant map $f$ at the origin depends on its symmetries as follows:

\begin{lemma}\label{LemaDerivada}
Let $f: \; V \rightarrow V$ be a $\Gamma$-equivariant map
differentiable at the origin. Then $Df(0)$, the Jacobian of $f$ at
the origin, commutes with $\Gamma$.
\end{lemma}

Straightforward calculations, Lemma~\ref{LemaDerivada} and Proposition \ref{Liapunov} allow us to describe the
Jacobian matrix at the origin and, from it, the local dynamics for maps equivariant under each of
the subgroups of  $O(2)$.

\begin{proposition} \label{proptable}
Let $f$ be a  planar map differentiable at the origin. The
admissible forms for the Jacobian matrix of $f$ at the origin are
those given in the third column of the Table in Appendix~\ref{EquivTable} depending on the symmetry
group of $f$. The fourth column of this table describes how
the dynamics near the origin depends on the symmetry of $f$.
\end{proposition}

The groups that do not possess a flip are $SO(2)$ and  $\Z_n$. 
For these groups, 
local dynamics does not determine global dynamics.
In order to complete our analysis we add some remarks about each case.

The dynamics of  an $SO(2)$-symmetric embedding is mostly determined by  its radial component,
as  can be seen by writing $f$ in polar coordinates as $f(\rho, \theta)=(R(\rho, \theta), T(\rho, \theta))$. 
Since $f$ is $SO(2)-$equivariant, then  for all $\alpha \in \R$ 
$f(\rho, \theta + \alpha)=(R(\rho, \theta), T(\rho, \theta)+\alpha)$.
Therefore,
$
f(\rho, \theta-\theta)=(R(\rho, 0), T(\rho, 0))=(R(\rho, \theta), T(\rho, \theta)-\theta).
$
So $R(\rho, \theta)$ only depends on $\rho$ and $R\in Emb(\R^+)$.

Suppose  $0$ is a local attractor.
If $\Fix(R)=\{0\}$ then the origin is a global attractor by Lemma \ref{lemgdim1}.
Otherwise, the fixed points of the radial component are invariant circles for $f$ and therefore, knowledge about local dynamics does not contribute to the description of global dynamics\footnote{We are grateful to a referee for pointing this out.}.

In general if  $f$ is an $SO(2)$-symmetric embedding, then
$\R^2$ may be decomposed into 
$f$-invariant annuli with centre at the origin. 
Each annulus is either the union of $f$-invariant circles with $f$ acting as a rotation in each circle (corresponding to fixed points of $R$) or,  for all points $p$ in the annulus,  $\omega(p)$ is contained in the same connected component of the boundary of the annulus.

\medskip

For the group $\Z_n$
we may have a local attractor or repellor if $n\geq 3$ and a local atractor, repellor or saddle in the case of $\Z_2$.
For a local attractor, almost any global dynamics may be realised.
Examples of dissipative $\Z_n$-equivariant diffeomorphism with  a periodic orbit of period $n$
are given in \cite{SofisbeSzlenk} for all $n\ge 2$.
Alarc\'on  \cite{AlarconDenjoy}  proves the existence of a Denjoy map of the circle with symmetry group $\Z_n$.
This is used to construct orientation preserving homeomorphisms of the plane with symmetry group $\Z_n$
having the origin,
the unique fixed point, as an asymptotic local attractor.
Moreover, for these homeomorphisms there exists  a global attractor containing the origin  that is  a  compact and connected subset 
of $\R^2$  with zero  Lebesgue measure.

\medskip
Summarising, we have obtained  conditions on planar $\Gamma$-equivariant maps $f$ under which a local attractor/repellor is always a global attractor/repellor. This concerns only subgroups  $\Gamma$ containing a reflection, the only subgroups where this is possible.

\paragraph{Acknowledgements:}
The research of all authors at Centro de Matem\'atica da Universidade do Porto (CMUP)
 had financial support from
 the European Regional Development Fund through the programme COMPETE and
 from  the Portuguese Government through the Funda\c c\~ao para
a Ci\^encia e a Tecnologia (FCT) under the project PEst-C/MAT/UI0144/2011.
B. Alarc\'on was also supported from EX2009-0283 of the Ministerio de Educaci\'on (Spain) and grants MICINN-08-MTM2008-06065 and MICINN-12-MTM2011-22956 of the Ministerio de Ciencia e
Innovaci\'on (Spain).

\bigbreak
\noindent{Email addresses:}
\smallbreak

\noindent{B. Alarc\'on --- alarconbegona@uniovi.es}\\
\noindent{ S.B.S.D. Castro --- sdcastro@fep.up.pt}\\
\noindent{ I.S. Labouriau ---  islabour@fc.up.pt}

\newpage
\appendix

\section{Proof of Proposition \ref{Liapunov}} \label{ApendiceEstabilidade}

In order to prove Proposition~\ref{Liapunov}
we show that the hypotheses guarantee that $f$ is either a uniform contraction or expansion in a neighbourhood of the origin.
For this we need the following standart result, similar to a  Lemma in Chapter 9 \S 1
of Hirsch and Smale \cite{HirschSmale}.

\begin{lemma}\label{lemNorm}
Let $A: \; \R^N \rightarrow \R^N$ be linear and let
$\alpha$, $\beta \in \R$ satisfy
$$
0<\alpha < \left| \lambda\right| < \beta
$$
for all eigenvalues $\lambda$ of $A$. Then there exists a basis for $\R^N$ such that, in the corresponding inner product and norm,
\begin{equation}\label{NormaQuadrado}
\alpha \norm{x}  \leq \; \norm{Ax} \; \leq \beta \norm{x} \qquad\forall \; x \in \R^N.
\end{equation}
and moreover,
\begin{equation}\label{MajoraProduto}
\left|\langle Ax,y\rangle\right|\; \leq \beta \norm{x} \norm{y} \qquad \forall \; x,y \in \R^N.
\end{equation}
\end{lemma}

\begin{proof}[Proof of Proposition \ref{Liapunov}]
Write $f(x)=Df(0).x +r(x)$. If $f$ is linear then $r(x) \equiv 0$. Otherwise, using the norm of the previous lemma, since $f$ is differentiable at the origin, we know that
$$
\lim_{x \rightarrow 0} \frac{\norm{r(x)}}{\norm{x}}=0,
$$
that is,
$$
\forall \; \varepsilon >0 \quad\exists \; \delta >0 : \quad \norm{x} < \delta \Rightarrow \norm{r(x)} < \varepsilon \norm{x}.
$$

\bigbreak

Assume the eigenvalues of $Df(0)=A$ have absolute value smaller than  $\beta<1$. Then, we have,
if $\norm{x} < \delta$:
\begin{eqnarray*}
 \norm{f(x)}^2  & = & \langle f(x),f(x)\rangle=\langle Ax+r(x), Ax+r(x)\rangle \\
 & = & \langle Ax,Ax\rangle + 2 \langle Ax, r(x)\rangle +\langle r(x) , r(x) \rangle  \\
 & \leq & \beta^2 \norm{x}^2 + 2 \beta \norm{x}\norm{r(x)} +  \norm{r(x)}^2\\
 & \leq & \beta^2 \norm{x}^2 + 2 \varepsilon\beta \norm{x}^2 + \varepsilon^2 \norm{x}^2 \\
 & \leq &\left( \beta^2  + 2 \varepsilon\beta  + \varepsilon^2\right) \norm{x}^2
\end{eqnarray*}
showing that $f$ is a uniform contraction in a ball of radius $\delta$ around the origin provided that
$\beta^2 + 2 \beta \varepsilon + \varepsilon^2 <1$. Since $\beta$ is fixed and  smaller that one, such an $\varepsilon$ always exists.

\bigbreak
Assume now that
$$
1 < \alpha <\left|\lambda\right| < \beta,
$$
for all eigenvalues $\lambda$ of $Df(0)=A$. Then
\begin{eqnarray*}
\norm{f(x)}^2& = & \langle f(x),f(x)\rangle = \langle Ax+r(x), Ax+r(x)\rangle  \\
 & = & \langle Ax,Ax \rangle + 2 \langle Ax, r(x) \rangle + \langle r(x) , r(x) \rangle  \\
 & \geq & \alpha^2 \norm{x}^2 + 2 \langle Ax, r(x) \rangle
\end{eqnarray*}
since $\norm{r(x)}^2 > 0$. It then follows that
\begin{eqnarray*}
 \norm{f(x)}^2 & \geq & \alpha^2 \norm{x}^2 + 2 \langle Ax, r(x) \rangle  \\
 & \geq &\alpha^2 \norm{x}^2 - 2\left| \langle Ax, r(x)\rangle \right| \\
 & \geq & (\alpha^2 - 2\beta \varepsilon) \norm{x}^2,
\end{eqnarray*}
showing that $f$ is a uniform expansion in a ball of radius $\delta$ around the origin provided that
$$
\varepsilon < \frac{\alpha^2 -1}{2\beta}.
$$
Since $\alpha^2 -1>0$ such an $\varepsilon$ always exists.

\end{proof}

\section{Compact subgroups of $O(2)$} \label{subgroups}

\begin{itemize}
    \item  $O(2)$, acting on $\R^2 \simeq \C$ as the group generated by $\theta$ and $\kappa$ given by
    $$
    \theta . z = e^{i\theta} z, \quad \theta \in S^1 \qquad\mbox{ and   } \qquad \kappa . z=\bar{z}.
    $$
    \item  $SO(2)$, acting on $\R^2 \simeq \C$ as the group generated by $\theta$  given by
    $$
    \theta . z = e^{i\theta} z, \quad \theta \in S^1.
    $$
    \item  $D_n$, $n  \geq 2$, acting on $\R^2 \simeq \C$ as the finite group generated by $\zeta$ and $\kappa$ given by
    $$
    \zeta . z = e^{\frac{2\pi i}{n}} z  \qquad\mbox{ and   } \qquad \kappa . z=\bar{z}.
    $$
    \item  $\Z_n$, $n  \geq 2$, acting on $\R^2 \simeq \C$ as the finite group generated by $\zeta$ given by
    $$
    \zeta . z = e^{\frac{2\pi i}{n}} z.
    $$
    \item  $\Z_2(\langle\kappa\rangle)$, acting on $\R^2$ as
    $$
    \kappa . (x,y) = (x, -y).
    $$

\end{itemize}

\section{The equivariant table}\label{EquivTable}

This appendix contains a table summarizing the bulk of results in the paper. The first column concerns the group of symmetry. The second column provides information about the existence of an invariant ray. The third and fourth columns concern the dynamics by providing the form of the jacobian matrix at the origin and a list of possible local dynamics. Finally, the last column lists hypotheses required to go from local to global dynamics.

No conditions are provided when the fixed point is a local saddle since this case is not addressed here.

\newpage
\thispagestyle{empty}
\begin{table}[hh] \label{Apendiceequvtable} \label{ApendiceTabela}
\centering

\begin{sideways}
\begin{tabular}{|c|c|c|c|c|c|}

\hline \multicolumn{5}{|c|}{EQUIVARIANT TABLE} \\

\hline Symmetry Group & Contains $\kappa$?& $Df(0)$ &  \multicolumn{1}{|p{3cm}|}{Hyperbolic \par Local Stability} & \multicolumn{1}{|p{5cm}|}{Hypothesis for Hyperbolic \par Global Stability.} \\

\hline $O(2)$ & yes  & $\left( \begin{matrix} \alpha & 0 \\ 0 & \alpha \end{matrix} \right)$ $\alpha \in \R$
& \multicolumn{1}{|p{3cm}|}{attractor \par $ $ \par repellor} & \multicolumn{1}{|p{5cm}|}{$\Emb(\R^2)$ differentiable and dissipative. \par $\Emb(\R^2)$ differentiable.} \\

\hline $SO(2)$ & no  & $\left( \begin{matrix} \alpha & -\beta \\ \beta & \alpha \end{matrix} \right)$ $\alpha, \beta \in \R$
& attractor / repellor &  \multicolumn{1}{|p{5cm}|}{ Other configurations. }\\

\hline $D_n, \;n\geq 3,$ & yes  & $\left( \begin{matrix} \alpha & 0 \\ 0 & \alpha \end{matrix} \right)$ $\alpha \in \R$
& \multicolumn{1}{|p{3cm}|}{attractor \par $ $ \par repellor} & \multicolumn{1}{|p{5cm}|}{$\Emb(\R^2)$ differentiable and dissipative. \par $\Emb(\R^2)$ differentiable.}\\

\hline $\Z_n, \;n\geq 3$ & no & $\left( \begin{matrix} \alpha & -\beta \\ \beta & \alpha \end{matrix} \right)$ $\alpha, \beta \in \R$ & attractor / repellor &  \multicolumn{1}{|p{5cm}|}{ Other configurations. }\\

\hline $\Z_2(\langle\kappa\rangle)$ & yes & $\left( \begin{matrix} \alpha & 0 \\ 0 & \beta \end{matrix} \right)$ $\alpha, \beta \in \R$ & \multicolumn{1}{|p{3cm}|}{attractor \par $ $ \par repellor \par saddle} & \multicolumn{1}{|p{5cm}|}{$\Emb(\R^2)$ differentiable and dissipative. \par $\Emb(\R^2)$ differentiable. \par ---}\\

\hline $\Z_{2}$ & no  & any matrix  & \multicolumn{1}{|p{3.5cm}|}{attractor / repellor \par saddle} & \multicolumn{1}{|p{5cm}|}{ Other configurations.  \par ---}\\

\hline $D_{2}$ & yes &  $\left( \begin{matrix} \alpha & 0 \\ 0 & \beta \end{matrix} \right)$ $\alpha, \beta \in \R$
& \multicolumn{1}{|p{3cm}|}{attractor \par $ $ \par repellor \par saddle} & \multicolumn{1}{|p{5cm}|}{$\Emb(\R^2)$ differentiable and dissipative. \par $\Emb(\R^2)$ differentiable. \par ---}\\

\hline

\end{tabular}
\end{sideways}
\end{table}

\end{document}